\documentclass{amsart}
\usepackage{amssymb,amsmath,latexsym}
\usepackage{amsthm}
\usepackage{fontenc}
\usepackage{amssymb}

\numberwithin{equation}{section}

\newtheorem{theorem}{Theorem}[section]

\begin{document}
\author{Alexander E Patkowski}
\title{On Arithmetic Series involving the fractional part function}

\maketitle
\begin{abstract}We present new relationships between the work of H. Davenport and A. I. Popov. A new general formula involving the von Mangoldt function is presented, as well as a criteria for the Riemann Hypothesis. \end{abstract}



\section{Introduction and Main results} In [3] Davenport investigated some interesting arithmetic series involving the fractional part function $\{x\},$ and established interesting connections with special Fourier series. To emphasize one curious result in relation to special functions, we note the formula
\begin{equation}\sum_{n\ge1}\frac{\Lambda(n)}{n}(\{nx\}-\frac{1}{2})=-\frac{1}{\pi}\sum_{n\ge1}\frac{\log(n)}{n}\sin(2\pi nx).\end{equation}
(This has a special connection with Kummer's Fourier series for $\log(\Gamma(x))$ (see [1]), which we touch upon later in this section.) More modern work has been done by the authors [1,7], where we may find natural refinements as 
series involving $B_k(\{x\}),$ the periodic Bernoulli polynomials. \\* The purpose of this paper is to present further 
interesting results connecting the work of Popov [9] and Davenport [3]. Popov's result can be view as an integral operation applied to arithmetic series of the Davenport type (i.e. the left side of (1.1)), and applying a lower bound on the terms of the resulting series. Popov's formula is special because of 
its relationship with series related to the prime number theorem [4, pg.69]. \par We first present a new refinement to Popov's 
formula more akin to those given by [1,7] by adapting the first proof offered by Patkowski [8] to periodic Bernoulli polynomials. 
Then we offer some natural corollaries and a criterion for the Riemann Hypothesis (RH) in the last section. We follow standard 
notation, and set $\rho$ to be the non-trivial zeros of the Riemann zeta function [6, 11], $\zeta(s).$

\begin{theorem} Define $I_k(x):=\int_{0}^{x}B_k(\{t\})dt.$ Put $$H_k(s)=(-1)^{k-1}\binom{-s}{k}^{-1}\left(\zeta(s)+\frac{1}{1-s}+\sum_{j\ge1}^{k}\binom{-s}{j-1}\frac{B_j}{j}\right).$$
For $x>1,$ and natural numbers $k\ge1,$ we have
\begin{equation} \sum_{n>x}\frac{\Lambda(n)}{n^{k+1}}I_k(\frac{n}{x})=P_{k-1}(\log(x))x^{-k}+Q_k(x)+\sum_{\rho}x^{\rho-1-k}\frac{H_k(1-\rho)}{(k+1-\rho)}+\sum_{j\ge1}x^{2j-1-k}\frac{H_k(1+2j)}{(k+1+2j)},
\label{Th-1}\end{equation}
where $P_k(x)$ is a polynomial of degree $k.$ Here $x^{-1}P_0(\log(x))=\frac{\log(2\pi)-2}{2x},$ and $x^{-2}P_1(\log(x))=\frac{8-\gamma-3\log(2\pi)+\log(x)}{6x^2}.$
Furthermore, $Q_1(x)=0,$ $Q_k(x)=\sum_{2\le i\le k}c_ix^{i-1-k}\frac{\zeta'(i)}{\zeta(i)},$ when $k>1,$ and the $c_i$ are computable.

\end{theorem}

\begin{proof} We use the proof offered in [8] as a guide, but start with the general integral formula from Cohen [2, pg.32], with $\alpha=-s.$ We have for $k\ge1,$
\begin{equation}(-1)^{k-1}\binom{-s}{k}^{-1}\zeta(s)=-(-1)^{k-1}\binom{-s}{k}^{-1}\frac{1}{1-s} 
\label{Cohen} \end{equation}
$$-(-1)^{k-1}\binom{-s}{k}^{-1}\sum_{j\ge1}^{k}\binom{-s}{j-1}\frac{B_j}{j}+\int_{1}^{\infty}t^{-s-k}B_k(\{t\})dt.$$
Hence, put $$H_k(s)=(-1)^{k-1}\binom{-s}{k}^{-1}\left(\zeta(s)+\frac{1}{1-s}+\sum_{j\ge1}^{k}\binom{-s}{j-1}\frac{B_j}{j}\right),$$
So then Mellin inversion of  \eqref{Cohen}  gives for $1<x<\infty,$ $k>-\Re(s)+1,$ $\Re(s)=c,$
$$x^{-(k-1)}B_k(\{x\})=\frac{1}{2\pi i}\int_{(c)}x^sH_k(s)ds.$$
Now we can compute that $k\ge1,$ $c>1-k,$
\begin{equation}\int_{0}^{x}B_k(\{t\})dt=\frac{1}{2\pi i}\int_{(c)}x^{s+k}\frac{H_k(s)}{k+s}ds.
\label{integral-Bernoulli}
\end{equation}
We compute the residues from $0$ to $-(k-1)$ using Mathematica to find that
  $$Res_{s=-j}(H_k(s)x^{-s-k}/(s+k))=0, ~~j=1,2, \ldots, k-1.$$   
 By absolute convergence, we may replace $x$ with $n/x$ and invert the desired series involving the von Mangoldt function. This means we have the integral formula, for $-k<\Re(s)=d<1-k,$
\begin{equation}\sum_{n>x}\frac{\Lambda(n)}{n^{k+1}}I_k(\frac{n}{x})=\frac{1}{2\pi i}\int_{(d)}x^{-s-k}\frac{H_k(s)\zeta'(1-s)}{\zeta(1-s)(k+s)}ds,\end{equation}
since $\zeta'(1-s)/\zeta(1-s)$ converges uniformly in $-k<\Re(s)=d<1-k,$ for each $k\ge1.$ Replace $s$ by $1-s$ to get that, for $k<\Re(s)=e<1+k,$
\begin{equation}\frac{1}{2\pi i}\int_{(e)}x^{s-1-k}\frac{H_k(1-s)\zeta'(s)}{\zeta(s)(k+1-s)}ds,\end{equation}
which has a double pole at $s=1,$ simple poles at $s=2, 3, 4, \cdots, k.$ The poles remaining to the left are at the non-trivial zeros $s=\rho$ and trivial zeros $s=-2k.$ Hence we have,
\begin{equation} \sum_{n>x}\frac{\Lambda(n)}{n^{k+1}}I_k(\frac{n}{x})=Res_{s=1}(x^{s-1-k}\frac{H_k(1-s)\zeta'(s)}{\zeta(s)(k+1-s)})+\sum_{2\le j\le k}Res_{s=j}(x^{s-1-k}\frac{H_k(1-s)\zeta'(s)}{\zeta(s)(k+1-s)})  \end{equation}
$$+\sum_{\rho}x^{\rho-1-k}\frac{H_k(1-\rho)}{(k+1-\rho)}+\sum_{j\ge1}x^{2j-1-k}\frac{H_k(1+2j)}{(k+1+2j)}.$$
The result now follows upon choosing our $P_k(x)$ to be the residues at $s=1$ and our $Q_k(x)$ to be the residues between $2$ and $k.$

\end{proof}

We make some more observations related to the work of Segal [10]. Set 
$$S(x)=-\frac{1}{\pi}\sum_{n\ge1}\frac{\sin(2n\pi x)}{n},$$
and note that $S(x)=\{x\}-\frac{1}{2},$ if $x\neq[x],$ and $0$ if $x=[x].$ Further put $\dot{S}(x)=\int_{0}^{x}S(t)dt.$
This will help us relate Popov's work to that of Davenport, since Davenport's series mainly differ in the summation bounds. This difference appears to be the reason for the appearance of the series involving the non-trivial zero's $\rho,$ as can be seen below.

\begin{theorem} We have for $x>0,$
\begin{equation}\sum_{n>0}\frac{\Lambda(n)}{n^2}\dot{S}(\frac{n}{x})=\frac{1}{\pi^2}\sum_{n>0}\frac{\log(n)}{n^2}(\cos(2n\pi/x)-1),\end{equation}
and
\begin{equation}\sum_{n>0}\frac{\mu(n)}{n^2}\dot{S}(\frac{n}{x})=\frac{1}{\pi^2}(\cos(2\pi/x)-1).\end{equation}
\end{theorem}
 \begin{proof} From Segal [10] equation (3), recall $-1<l<0,$ $x>0,$
$$S(x)=-\frac{1}{2\pi i}\int_{(l)}\frac{x^s\zeta(s)}{s}ds.$$
Hence
$$\dot{S}(x)=-\frac{1}{2\pi i}\int_{(l)}\frac{x^{s+1}\zeta(s)}{s(s+1)}ds,$$ and therefore

$$\sum_{n>0}\frac{\Lambda(n)}{n^2}\dot{S}(\frac{n}{x})=-\frac{1}{2\pi i}\int_{(l)}\frac{x^{-s-1}\zeta(s)\zeta'(1-s)}{\zeta(1-s)s(s+1)}ds$$
$$=\frac{1}{2\pi^2 i}\int_{(l)}\frac{x^{-s-1}\zeta'(1-s)\Gamma(-s)\sin(\frac{\pi}{2} s)(2\pi)^s}{(s+1)}ds.$$
Replacing $s$ by $-s$ gives
$$\frac{1}{2\pi^2 i}\int_{(-l)}\frac{x^{s-1}\zeta'(1+s)\Gamma(s)\sin(\frac{\pi}{2} s)(2\pi)^{-s}}{(1-s)}ds=\frac{1}{2\pi^2 i}\int_{(-l)}x^{s-1}\Gamma(s-1)\zeta'(1+s)\sin(\frac{\pi}{2} s)(2\pi)^{-s}ds,$$
 replacing $s$ by $s+1$ gives, by absolute convergence,
$$\frac{1}{2\pi i}\int_{(-l-1)}x^{s}\Gamma(s)\zeta'(s+2)\cos(\frac{\pi}{2} s)(2\pi)^{-s-1}ds=\frac{1}{2\pi^2}\sum_{n>0}\frac{\log(n)}{n^2}(\cos(2n\pi/x)-1).$$

 Similar computations give (1.9).

 \end{proof}
We mention that there is a clear relationship between the right hand side of (1.8) and the integral of Kummer's Fourier expansion for $\log(\Gamma(x))$ (see [1] for similar observations) for $0<x<1$
$$\int_{0}^{x}\log(\Gamma(t))dt=\frac{x}{2}\log(2\pi)-\sum_{n\ge1}\frac{\log(2\pi)+\gamma+\log(n)}{2\pi^2 n^2}(\cos(2n\pi x)-1)+\frac{1}{4\pi}\sum_{n\ge1}\frac{\sin(2n\pi x)}{n^2}.$$
Equation (1.9) is actually a different form of a result that is given by H. Davenport [3, pg.11], and a refinement is given in [7, eq.(2.6)]. This theorem tells us that Popov's sum bound $n>x$ gives the relationship to the non-trivial zeros, which we believe makes it more interesting.
\par 
We also mention that we came across an alternative form of Euler-Maclaurin formula involving a contour integral during our computational work that may be of independent interest. Using [6, pg. 69] of Iwaniec and Kowalski,
let $g_k(s;a,b)=\int_{a}^{b}t^{s+k-1}f^{(k)}(t)dt,$ then
$$\frac{(-1)^k}{k!}\int_{a}^{b}f^{(k)}(t)B_k(\{t\})dt=\frac{(-1)^k}{k!2\pi i}\int_{(c)}H_k(s)g_k(s;a,b)ds$$
$$=\int_{a}^{b}f(t)dt-\sum_{a<n\le b}f(n)+\sum_{l\ge1}^{k}\frac{(-1)^l}{l!}(f^{(l-1)}(b)-f^{(l-1)}(a))B_l.$$
Here $f(t)\in C^{k}[a,b],$ the space of $k$ continuous derivatives in $[a,b].$
\section{A Criteria for the Riemann Hypothesis}
As an application of our investigation of arithmetic series, we show how to construct a special criterion for the Riemann Hypothesis using the analytic work from Segal, and the criterion of Littlewood.

\begin{theorem} Define $\bar{\mu}(n):=\sum_{d|n}\mu(d)\sqrt{d}\mu(\frac{n}{d}).$ We have,
\begin{equation}\sum_{n>0}\frac{\bar{\mu}(n)}{n^{2}}\dot{S}(\frac{n}{x})=O(x^{-\delta'-1}),\end{equation}
as $x\rightarrow\infty$ for every $\delta'<0.$ Furthermore, (2.1) holds unconditionally if $\delta'\le-\frac{1}{2},$ and under the RH when $-1/2<\delta'<0.$
\end{theorem}

\begin{proof}
Note that $\sum_{n\ge1}\bar{\mu}(n)n^{-s}=1/(\zeta(s)\zeta(s-\frac{1}{2})),$ which converges when $\Re(s)>1$ assuming the RH. Then, if we assume the Riemann hypothesis, and apply Littlewood's condition for the RH that $\sum_{n\ge1}\mu(n)n^{-1/2-\epsilon}$ coverges for every $\epsilon>0$ [1, pg.161], then $-1<l<0,$
$$\sum_{n>0}\frac{\bar{\mu}(n)}{n^{2}}\dot{S}(\frac{n}{x})=-\frac{1}{2\pi i}\int_{(l)}\frac{x^{-s-1}\zeta(s)}{\zeta(1-s)\zeta(\frac{1}{2}-s)s(s+1)}ds=\frac{1}{2\pi^2 i}\int_{(l)}\frac{x^{-s-1}\Gamma(-s)\sin(\frac{\pi}{2} s)(2\pi)^s}{\zeta(\frac{1}{2}-s)(s+1)}ds.$$
This is certainly valid when $-1/2\ge\Re(s)=l>-1,$ without the assumption of the RH, but requires the RH when $-1/2<\Re(s)=l<0.$ Hence assuming the RH we have
\begin{equation}\sum_{n>0}\frac{\bar{\mu}(n)}{n^{2}}\dot{S}(\frac{n}{x})=O(x^{-\delta-1}),\end{equation}
where $-1/2<\delta<0.$ Combining these observations, it follows that we obtain the theorem.
\end{proof}

Next we consider a result similar to our series identities but with the added assumption of the RH.
\begin{theorem} Define $\Upsilon(n):=\sum_{d|n}\mu(d)\sqrt{d}.$ For $x>1,$ assuming the RH we have that
\begin{equation}\sum_{n>0}\frac{\mu(n)}{n^{3/2}}\dot{S}(\frac{n}{x})=\frac{1}{\pi^2}\sum_{n>0}\frac{\Upsilon(n)}{n^2}(\cos(2\pi n/x)-1).\end{equation}
\end{theorem}
\begin{proof} If we assume the RH, and again apply Littlewood's condition, then 
$$\sum_{n>0}\frac{\mu(n)}{n^{3/2}}\dot{S}(\frac{n}{x})=-\frac{1}{2\pi i}\int_{(l)}\frac{x^{-s-1}\zeta(s)}{\zeta(\frac{1}{2}-s)s(s+1)}ds=\frac{1}{2\pi^2 i}\int_{(l)}\frac{x^{-s-1}\zeta(1-s)\Gamma(-s)\sin(\frac{\pi}{2} s)(2\pi)^s}{\zeta(\frac{1}{2}-s)(s+1)}ds,$$
where we have applied the functional equation for the Riemann zeta function.
Now replace $s$ by $-s$ to get that
$$\frac{1}{2\pi^2 i}\int_{(-l)}\frac{x^{s-1}\zeta(1+s)\Gamma(s)\sin(\frac{\pi}{2} s)(2\pi)^{-s}}{\zeta(\frac{1}{2}+s)(1-s)}ds=\frac{1}{2\pi^2 i}\int_{(-l)}\frac{x^{s-1}\zeta(1+s)\Gamma(s-1)\sin(\frac{\pi}{2} s)(2\pi)^{-s}}{\zeta(\frac{1}{2}+s)}ds,$$
and replace $s$ by $s+1$ to get
$$\frac{1}{2\pi^2 i}\int_{(-l-1)}x^{s}\Gamma(s)\cos(\frac{\pi}{2} s)(2\pi)^{-s-1}\frac{\zeta(2+s)}{\zeta(\frac{3}{2}+s)}ds.$$
We now use the multiplicative convolution with $\Upsilon(n):=\sum_{d|n}\mu(d)\sqrt{d}.$ Thus, $\sum_{n\ge1}\Upsilon(n)n^{-s}=\zeta(s)/\zeta(s-\frac{1}{2}),$ which coverges when $\Re(s)>1$ assuming the RH. Interchanging this series with the integral gives the result.

\end{proof}

1390 Bumps River Rd. \\*
Centerville, MA
02632 \\*
USA \\*
E-mail: alexpatk@hotmail.com, alexepatkowski@gmail.com

\end{document}